\pdfoutput=1
\documentclass[conference]{IEEEtran}
\IEEEoverridecommandlockouts

\usepackage[T1]{fontenc}
\usepackage[utf8]{inputenc}
\usepackage{amsmath}
\usepackage{amssymb}
\usepackage{amsthm}
\usepackage{graphicx}
\usepackage{booktabs}
\usepackage{array}
\usepackage{url}
\usepackage[ruled,vlined,linesnumbered]{algorithm2e}

\newcommand{\TS}{\textsc{TwoSum}}
\newcommand{\FTS}{\textsc{FastTwoSum}}
\newcommand{\VS}{\textsc{VecSum}}
\newcommand{\VSK}{\textsc{VecSum$K$}}
\newcommand{\RP}{\textsc{RenormBF-pair}}
\newcommand{\RN}{\mathrm{RN}}
\newcommand{\ulp}{\mathrm{ulp}}
\newcommand{\vx}{\mathbf{x}}
\newcommand{\vb}{\mathbf{b}}
\newcommand{\vr}{\mathbf{r}}
\newcommand{\vp}{\mathbf{p}}
\newcommand{\vq}{\mathbf{q}}
\newcommand{\vs}{\mathbf{s}}

\newtheorem{proposition}{Proposition}
\newtheorem{lemma}{Lemma}
\newtheorem{observation}{Observation}
\theoremstyle{remark}
\newtheorem{remark}{Remark}

\begin{document}

\title{A Branch-Free General Renormalization Scheme\\for Pair Arithmetic
and Its Performance Evaluation}

\author{\IEEEauthorblockN{Tomonori Kouya}
\IEEEauthorblockA{Faculty of Science and Engineering\\
Otemon Gakuin University, Ibaraki, Osaka, Japan\\
Email: t-koya@haruka.otemon.ac.jp}}

\maketitle

\begin{abstract}
Pair arithmetic omits the renormalization stage performed at the end of each
operation of multi-component multiple-precision arithmetic, reducing the
operation count and, by eliminating conditional branches, easing SIMD
vectorization. Renormalization cannot be omitted entirely, however, because
iterative solvers may then fail to converge. We propose RenormBF-pair, a branch-free
fixed-trip-count renormalizer parameterized by the word count \(K\) and the
input length \(n\): a VecSum sweep, a tail fold, and \(r\) rounds of a
FastTwoSum chain, costing \(6(n-1)+(n-K)+3r(K-1)\) flops. Over \(200{,}000\)
trials per condition, the smallest round count producing no non-overlap
violation was \(r=1\) for \(K \leq 3\) and \(r=2\) for \(K=4\), whereas the same
cost spent on a VecSum\(K\)-style TwoSum chain fails for \(K=4\). We
delimit with explicit counterexamples both the input condition under which the
sum relation holds and the fact that the scheme does not, in general, guarantee
a correctly rounded lowest component. Integrating it into CG and BiCGStab over pair arithmetic
in binary64 and binary32, we sweep \(240\) paired runs including real matrices
screened from SuiteSparse. The proposed scheme and VecSum\(K\) returned bitwise
identical words wherever both succeeded, so they are indistinguishable for CG in
binary64; differences appear only where VecSum\(K\) breaks, namely binary32 and
BiCGStab. There the \(22\)--\(16\) win--loss record is not significant (sign test,
\(p=0.42\)); the clear bias is structural, non-overlap violations dropping from
\(66\) rows to \(37\). The cost is \(+3.4\%\) for the whole CG solver.
\end{abstract}

\begin{IEEEkeywords}
multiple-precision floating-point arithmetic, pair arithmetic, quasi
multi-word, renormalization, error-free transformation, Krylov subspace
methods, SIMD
\end{IEEEkeywords}

\section{Introduction}

Standard IEEE~754 binary32 and binary64 arithmetic may be insufficient to
maintain numerical stability in scientific computations; Krylov subspace
methods provide a typical example. Higher-precision arithmetic may then be
required, and one way to implement it is multi-component arithmetic, which
represents a number as a sum of several binary32 or binary64 values and builds
the four arithmetic operations out of error-free transformations.

Multi-component (multi-word) arithmetic, which represents a $K$-word number as
the sum $z = z_0 + z_1 + \cdots + z_{K-1}$ of $K$ floating-point values, is
widely used as double-word ($K=2$), triple-word ($K=3$) and quad-word ($K=4$)
arithmetic \cite{HidaLiBailey01,JoldesMullerPopescu18,FabianoMullerPicot19}.
Throughout, $p$ denotes the significand precision in bits ($53$ for binary64,
$24$ for binary32) and $u=2^{-p}$ the unit roundoff. Every operation of this
scheme ends with a renormalization stage that makes the output
\emph{non-overlapping},
\begin{equation}
|z_{k+1}| \;\le\; \tfrac12\ulp(z_k) \qquad (k = 0,\dots,K-2).
\label{eq:nonoverlap}
\end{equation}
That stage is a \FTS{} chain: it carries no conditional branch, but it does
create a chain of sequential dependences inside the operation.

Pair arithmetic omits this stage. The pair arithmetic of Lange and Rump
\cite{LangeRump20} ($K=2$) and the extension of Ozaki and Imamura
\cite{OzakiImamura23} ($K=3,4$) are of this kind, and Mukunoki and Ozaki
\cite{MukunokiOzaki25} called it quasi multi-word arithmetic (QDW/QTW/QQW) and
applied it to sparse iterative solvers. The operation counts drop, e.g.\ DWadd
$11 \to$ QDWadd $8$ and DWmul $7 \to$ QDWmul $4$, and removing the dependence
chain also improves SIMD efficiency. The author confirmed this earlier for
double-word BLAS1 kernels without renormalization and their application to
explicit extrapolation methods \cite{Kouya19}.

Renormalization cannot, however, be omitted completely. Mukunoki and Ozaki
\cite{MukunokiOzaki25} report that ``no convergence was achieved without
normalization,'' and insert one renormalization per iteration, immediately
after the AXPY that updates the residual vector $\vr$. The renormalizer is
\FTS{} for $K=2$ and \textsc{VecSum3} for $K=3$
($[c_1,c_2]\leftarrow\TS(c_1,c_2)$ followed by
$[c_2,c_3]\leftarrow\TS(c_2,c_3)$). The same paper states that
\textsc{VecSum3} is not a strict renormalization, and lists the optimal
frequency and placement of renormalization as an open problem.

This paper is about the \emph{content} of that single renormalization: rather
than its frequency or placement, we ask how strong one renormalization can be
made under a branch-free fixed-trip-count constraint. The proposed \RP{} is a
single construction parameterized by $K$ and $n$ and accepts any $n\ge K$
(Section~\ref{sec:alg}). A sum relation is given as a proposition, but its
hypothesis depends on the exponent layout of the input and is not established
in general by one sweep (counterexample in Remark~\ref{rem:notorder}).
\RP{} also does not, in general, guarantee a correctly rounded lowest
component, for a reason isomorphic to the Table Maker's Dilemma; the related
limitation of sweep-based schemes is discussed in
Section~\ref{sec:guarantee}. Section~\ref{sec:num}
determines the required configuration experimentally and compares it against a
cost-equivalent \VSK{}-style chain, and Section~\ref{sec:bench} plugs the scheme
into CG and BiCGStab over pair arithmetic in binary64 and binary32, sweeping
$240$ paired runs including real matrices screened from SuiteSparse. The results do
not yield a simple ranking: the two renormalizers returned identical words
wherever both succeeded, so they cannot be told apart for CG in binary64, and
differences appear only where \VSK{} actually breaks.

\section{A Branch-Free General Renormalizer for Pair Arithmetic}
\label{sec:alg}

\subsection{Problem setting}

The input is $n$ words $v_0,\dots,v_{n-1}$ whose sum $\sum_i v_i$ equals the
target value, i.e.\ the intermediate result of a pair-arithmetic operation. The
words are not necessarily sorted by magnitude and may overlap. The goal is to
return $K$ words satisfying the non-overlap condition \eqref{eq:nonoverlap};
whether that condition is achieved is evaluated experimentally in
Section~\ref{sec:num}.

A \VSK{}-style \TS{} chain does not accept this setting: it fixes the input
length to $n=K$ and implicitly expects the input to be already sorted by
magnitude. Calling the proposed scheme \emph{general} means that a single
construction parameterized by $K$ and $n$ handles any $n \ge K$. It does
\emph{not} mean that the output value is guaranteed for an input in arbitrary
order. The hypothesis required by step~(2) depends on the exponent layout of
the input and is not established in general by one sweep
(Section~\ref{sec:guarantee}, Remark~\ref{rem:notorder}). What can be stated
mathematically is only the conditional claim ``if the exponent condition holds
at every gate, then Proposition~\ref{prop:sum} holds''; the inputs we target
are those generated by pair-arithmetic networks, that is, input~[II] of
Section~\ref{sec:num}.

\subsection{The proposed algorithm}

\begin{algorithm}[t]
\caption{$\RP(v, n, K, r)$ --- branch-free general renormalization of a
pair-arithmetic output}
\label{alg:renormpair}
\SetKwInOut{Input}{Input}\SetKwInOut{Output}{Output}
\Input{$v_0,\dots,v_{n-1}$ (no overlap assumption; on ordering see
       Remark~\ref{rem:notorder}), word count $K$, number of chain rounds
       $r$}
\Output{$K$ words $z_0,\dots,z_{K-1}$ (non-overlap is evaluated in
        Section~\ref{sec:num})}
\BlankLine
\tcp{(0) \VS{} sweep: redistribute components toward the leading
     positions \hfill $6(n-1)$ flops}
\For{$i \leftarrow n-2$ \KwTo $0$}{
  $(v_i, v_{i+1}) \leftarrow \TS(v_i, v_{i+1})$\;
}
\BlankLine
\tcp{(1) tail fold: fold the discarded words into $v_{K-1}$, from the
     end of the array \hfill $n-K$ flops}
\If{$n > K$}{
  $t \leftarrow v_{n-1}$\;
  \For{$i \leftarrow n-2$ \KwTo $K$}{
    $t \leftarrow \mathrm{fl}(t + v_i)$ \tcp*{plain addition}
  }
  $v_{K-1} \leftarrow \mathrm{fl}(v_{K-1} + t)$\;
}
\BlankLine
\tcp{(2) $r$ rounds of the \FTS{} chain \hfill $3r(K-1)$ flops}
\For{$q \leftarrow 1$ \KwTo $r$}{
  \For{$i \leftarrow K-2$ \KwTo $0$}{
    $(v_i, v_{i+1}) \leftarrow \FTS(v_i, v_{i+1})$\;
  }
}
\BlankLine
\Return $v_0,\dots,v_{K-1}$\;
\end{algorithm}

Algorithm~\ref{alg:renormpair} gives the proposed \RP{}. Its cost is
\begin{equation}
C_{\RP}(n,K,r) \;=\; 6(n-1) \;+\; (n-K) \;+\; 3r(K-1)
\label{eq:cost}
\end{equation}
flops and depends only on $n$, $K$ and $r$. Since $n$, $K$ and $r$ are
compile-time constants, the \texttt{if} of step~(1) disappears under
specialization, and neither a run-time branch nor a variable trip count
remains. The smallest round count that produced no violation in the
measurements of Section~\ref{sec:num} was
\begin{equation}
r \;=\; \begin{cases} 1 & (K \le 3) \\ 2 & (K = 4) \end{cases}
\label{eq:rchoice}
\end{equation}
This is the minimum over the range we tested, not a proof of sufficiency.

The three stages play the following roles.

\paragraph*{(0) \VS{} sweep}
The \FTS{} of step~(2) requires a hypothesis
(Section~\ref{sec:guarantee}). A pair-arithmetic output is not even sorted by
magnitude, so the hypothesis does not hold as is. One sweep tends to move
larger components toward the leading positions and produces a layout in which
the hypothesis is more likely to hold. It cannot be omitted: without it \FTS{}
ceases to be an error-free transformation and the sum relation
(Proposition~\ref{prop:sum}) collapses (Section~\ref{sec:num}; roughly
$18{,}200$ out of $600{,}000$ calls for $K=3,4$). It does \emph{not}, however,
sort the components, so it does not guarantee the hypothesis in general
(Remark~\ref{rem:notorder}).

\paragraph*{(1) tail fold}
The $n-K$ words outside the $K$ retained ones are folded into the last
retained word $v_{K-1}$, starting from the end of the array. Since the sweep
tends to leave smaller components toward the end, this ordering tends to keep
the rounding error of the running sum small, although the tail is not guaranteed
to be sorted. When $n = K$ there is nothing to fold, so this step must be empty.

\begin{remark}[the $n=K$ guard cannot be omitted]
\label{rem:guard}
Without the guard, $t \leftarrow v_{n-1}$ becomes $t = v_{K-1}$, and the
assignment $v_{K-1} \leftarrow \mathrm{fl}(v_{K-1}+t)$ after the empty loop
\emph{doubles the lowest word}. Equation~\eqref{eq:cost} counts step~(1) as $0$
flops for $n=K$, so only with the guard do the cost formula and the
implementation agree. A subtle issue is that the error slips through the non-overlap test: for $K=2$ the output stays non-overlapping (violations
$0/200{,}000$) while the value is wrong. The worst measured
$\rho = |\sum_k z_k - S|/(u^K|S|)$ reaches $2^{87}$ for $K=2$, $2^{120}$ for
$K=3$ and $2^{153}$ for $K=4$.
\end{remark}

\paragraph*{(2) \FTS{} chain}
This renormalizes the $K$ retained words from the bottom upward. One round
costs $3(K-1)$ flops, half of the $6(K-1)$ flops of the \VSK{}-style \TS{}
chain (top down) that plays the same role. One round is not enough for $K=4$
(Section~\ref{sec:num}).

\subsection{Relation to \VSK{}}

\begin{observation}
\label{obs:degenerate}
Degenerating \RP{} to $n=K$ (no tail), replacing \FTS{} by \TS{} and dropping
step~(0) yields the $K-1$ stage \TS{} chain over the $K$ retained words, that
is, \VSK{}.
\end{observation}

The statement in \cite{MukunokiOzaki25} that \textsc{VecSum3} is not a strict
renormalization is explained by Observation~\ref{obs:degenerate}: for $n=K$
there is no tail to fold, so the overlap can be \emph{reduced but not removed}.

\section{Properties of \RP{}}
\label{sec:guarantee}

$\FTS(a,b)$ is defined by $s \leftarrow \mathrm{fl}(a+b)$,
$e \leftarrow \mathrm{fl}(b - \mathrm{fl}(s-a))$, and a sufficient condition for
$a+b = s+e$ to hold \emph{exactly} (i.e.\ for it to be an error-free
transformation) is the exponent condition $e_a \ge e_b$. The frequently quoted
$|a| \ge |b|$ is only sufficient for the exponent condition, not necessary.
This distinction matters below.

\begin{lemma}[\TS{} is unconditional]
\label{lem:ts}
The \TS{} of Knuth and M\o ller \cite{ORO05} satisfies $s+e = a+b$ exactly for
any inputs $a,b$, with no assumption (barring overflow).
\end{lemma}

\begin{lemma}[\FTS{} under the exponent condition]
\label{lem:fts}
If $a=0$ or $e_a \ge e_b$, then $s-a$ is exactly representable as a
floating-point number, so $\mathrm{fl}(s-a)=s-a$. Moreover $b-(s-a)=a+b-s$ is
the rounding error of $a+b$ and is representable as well, hence $e = a+b-s$,
that is, $s+e=a+b$ holds exactly \cite{Muller18}.
\end{lemma}

\begin{proposition}[sum relation]
\label{prop:sum}
Let $S = \sum_{i} v_i$ be the exact sum of the input. If the hypothesis of
Lemma~\ref{lem:fts} holds at every \FTS{} of step~(2), then
\begin{equation}
S \;=\; \sum_{k=0}^{K-1} z_k \;+\; \sum_{j=1}^{n-K} \delta_j
\label{eq:sumpres}
\end{equation}
holds exactly, where $\delta_j$ is the rounding error of the $j$-th addition of
the tail fold. In particular, for $n=K$ the identity $\sum_k z_k = S$ is exact.
\end{proposition}

\begin{IEEEproof}
Both \TS{} and \FTS{} replace an adjacent pair $(v_i,v_{i+1})$ of the array by
$(s,e)$, so it suffices to follow the real-valued total $T(v)=\sum_i v_i$.
By Lemma~\ref{lem:ts} each replacement of step~(0) leaves $T$ unchanged, so the
array $w$ obtained right after it satisfies $\sum_i w_i = S$; no hypothesis is
needed, so this holds however much the input overlaps. In step~(1),
$t_0=w_{n-1}$ and
$t_j=\mathrm{fl}(t_{j-1}+w_{n-1-j})=t_{j-1}+w_{n-1-j}-\delta_j$ give, by
induction, $t_{n-1-K}=\sum_{i=K}^{n-1}w_i-\sum_{j=1}^{n-1-K}\delta_j$. Adding
the final addition $u_{K-1}=w_{K-1}+t_{n-1-K}-\delta_{n-K}$ and the untouched
$u_k=w_k\ (k\le K-2)$ yields
$\sum_{k<K}u_k = S-\sum_{j\le n-K}\delta_j$; the number of additions is exactly
$n-K$, matching the second term of \eqref{eq:cost}. Under the hypothesis of
Lemma~\ref{lem:fts}, step~(2) leaves $\sum_{k<K}v_k$ invariant. For $n=K$,
step~(1) is empty by the guard of Remark~\ref{rem:guard}, so no $\delta_j$
exists.
\end{IEEEproof}

The error analysis therefore reduces to bounding the $n-K$ terms $\delta_j$;
steps~(0) and~(2) lose nothing as long as the hypothesis of
Lemma~\ref{lem:fts} holds.

\begin{remark}[one sweep does not guarantee the hypothesis]
\label{rem:notorder}
The \VS{} sweep of step~(0) does not in general sort the components by
magnitude, and it does not even guarantee that the leading word $v_0$ is the
largest: in binary64 the input $(-1,\,1,\,2^{-53})$ becomes $(0,\,0,\,2^{-53})$
after the sweep. In particular it does not guarantee the exponent condition
required by the subsequent \FTS{} chain, so the hypothesis of
Lemma~\ref{lem:fts} does not hold in general and inputs exist for which the
assumption of Proposition~\ref{prop:sum} fails. For a case in which the sum
itself is lost, take $K=n=3$ and
$(v_0,v_1,v_2)=(2^{-100},\,2^{53},\,1)$. After the sweep the
array is $(2^{53},\,2^{-100},\,1)$, which is not sorted by magnitude, and the
first $\FTS(2^{-100},1)$ of step~(2) violates the hypothesis because
$e_a<e_b$. Indeed $s=\mathrm{fl}(2^{-100}+1)=1$ and $\mathrm{fl}(s-a)=1$, so the
error term is $0$ and $2^{-100}$ vanishes. The output $(2^{53},1,0)$ is
non-overlapping but does not preserve the exact sum $2^{53}+1+2^{-100}$.
\textbf{The non-overlap test does not detect this error.} The inputs we target
are those generated by pair-arithmetic networks; for input~[II] of
Section~\ref{sec:num}, the hypothesis of Proposition~\ref{prop:sum} held at all
$600{,}000$ \FTS{} calls and error-freeness never broke. This is a measurement,
not a proof covering every input a pair-arithmetic network can produce.
\end{remark}

The quantity we want, $\tau = \sum_i v_i$, is a sum of finitely many
floating-point numbers, not an irrational number, and is exactly represented.
An obstruction isomorphic to the Table Maker's Dilemma \cite{Muller18} of
correctly rounded elementary functions nevertheless appears.

Obtaining $z_0 = \RN(\tau)$ is easy; the lowest word is the problem. Rounding
componentwise correctly means requiring
$z_{K-1} = \RN(\tau - \sum_{k<K-1}z_k)$, and deciding this requires knowing on
which side of the rounding boundary $\tfrac12\ulp(v_{K-1})$ the discarded tail
lies. From non-overlap $|v_{j+1}|\le\tfrac12\ulp(v_j)$ and
$\ulp(x)\le 2u|x|$ for normal numbers we obtain $|v_{j+1}|\le u|v_j|$ (our
target inputs contain no underflow, so subnormals do not enter this
derivation). All that follows for the tail is therefore the geometric bound
\begin{equation}
|v_K + v_{K+1} + \cdots + v_{n-1}| \;\le\;
   \frac{\tfrac12\ulp(v_{K-1})}{1-u},
\label{eq:half}
\end{equation}
whose right-hand side \emph{exceeds} the rounding boundary
$\tfrac12\ulp(v_{K-1})$ by the factor $1/(1-u)$. Non-overlap thus confines the
tail to a neighbourhood of the boundary but \textbf{does not decide which side
of it the tail is on}: exactly the one bit needed for correct rounding is
missing.

\begin{remark}
One cannot conclude $|v_K+\cdots+v_{n-1}|\le\tfrac12\ulp(v_{K-1})$ for the whole
tail from pairwise non-overlap. The values $v_{K-1}=1$, $v_K=2^{-53}$ and
$v_{K+1}=2^{-106}$ satisfy pairwise non-overlap with equality, yet
$v_K+v_{K+1}>2^{-53}=\tfrac12\ulp(1)$. The factor $1/(1-u)$ in
\eqref{eq:half} is needed to absorb this.
\end{remark}

\paragraph*{Counterexample}
We give a concrete counterexample for $K=2$, $n=10$. The input is the exact
$10$-word expansion, by error-free transformations, of $\tau = xy+c$ with
non-overlapping double-word $x$, $y$ and $c$.

\begin{center}\scriptsize
\begin{tabular}{@{}r@{\ }l@{\ }l@{}}
\toprule
$x$ & $=($ & \texttt{-0x1.cdc2df1bbe8b7p+20}, \\
    &      & \ \texttt{0x1.4e86b2fb73eeap-282}$)$ \\
$y$ & $=($ & \ \texttt{0x1.9eec2b8d61b11p+12}, \\
    &      & \ \texttt{0x1.e6a17ec1b3876p-286}$)$ \\
$c$ & $=($ & \ \texttt{0x1.76359181a5defp+33}, \\
    &      & \ \texttt{0x1.42d23dfff824ep-21}$)$ \\
\midrule
\multicolumn{3}{@{}l@{}}{after 3, 4, 6, 12 or 24 sweeps, all give:} \\
$z_1$ & $=$ & \ \texttt{0x1.c9c513651457ap-320} \\
\multicolumn{3}{@{}l@{}}{exact $\RN(\tau - z_0)$:} \\
$z_1$ & $=$ & \ \texttt{0x1.c9c513651457bp-320} \\
\bottomrule
\end{tabular}
\end{center}

For this input, even after $24$ sweeps have fully restored pairwise
non-overlap,
\[
 \frac{|v_2+\cdots+v_9|}{\tfrac12\ulp(v_1)} \;=\; 1 + 4.95\times10^{-59},
\]
so the tail exceeds the rounding boundary by an \emph{extremely} small amount.
The excess is of order $2^{-194}$, far inside the margin
$u/(1-u)\approx1.1\times10^{-16}$ allowed by \eqref{eq:half}. In other words,
\textbf{non-overlap cannot decide this one bit}. Adding sweeps always yields a
bound of the form \eqref{eq:half}, never the information of which side of $1$
the ratio lies on. Since the sweep count is precisely the degree of freedom of
a fixed-trip-count design, there is no solution in the direction of increasing
it.

\begin{proposition}[lowest word of \RP{}]
\label{prop:notcr}
\RP{} does not in general guarantee
$z_{K-1} = \RN(\tau - \sum_{k<K-1} z_k)$ for the lowest word $z_{K-1}$.
\end{proposition}

Since the counterexample above yields the same wrong $z_1$ irrespective of $r$
and of the sweep count, that single instance proves the proposition.

\begin{remark}
The same difficulty is expected to extend to schemes in general that perform a
fixed number of \TS{}/\FTS{} sweeps and retain no sticky information about the
discarded tail, but we have not proved impossibility for that whole family, let
alone for the branch-free fixed-trip-count model of computation as such. With
$n$ fixed, a construction that accumulates exactly into a sufficiently long
fixed-size accumulator and rounds with a fixed number of bit operations is
conceivable in principle. The analogy with the Table Maker's Dilemma is an
explanation, not a proof.
\end{remark}

Representative remedies are round-to-odd \cite{BoldoMelquiond08}, which
preserves a sticky bit when the tail is discarded, and a data-dependent
iteration such as \textsc{NearSum} \cite{ORO05}, the summation counterpart of
Ziv's strategy. The former is absent from the binary64 hardware rounding modes
and the latter has an input-dependent iteration count, so neither is compatible
with the sweep-based branch-free fixed-trip-count constraint adopted here.

\begin{remark}
\RP{} therefore does not claim correct rounding of the lowest word. Its
practical goal is to fold the discarded tail into the last retained word and
restore non-overlap; the latter is evaluated experimentally in
Section~\ref{sec:num}. For $n>K$, faithfulness is not guaranteed either unless a
bound on $\sum_j\delta_j$ is supplied. For the purpose of pair arithmetic,
namely resolving the overlap enough for an iterative solver to converge, this is
sufficient.
\end{remark}

\section{Numerical Verification}
\label{sec:num}

\subsection{Generated data and verification method}

Evaluating a renormalizer requires feeding it the ``overlapping $n$ words'' that
pair arithmetic can actually produce. We generated the following two families of
random input.

\begin{itemize}
\item \emph{[I] Synthetic model}: starting from the exponent of the leading
  word, $n$ words are produced with the exponent lowered by a constant gap $g$
  at each step. At $g=53$ the words are nearly non-overlapping, and the smaller
  $g$ is, the deeper the overlap. Sweeping $g\in\{53,42,31,20\}$ probes
  directly how deep an overlap a renormalizer tolerates.
\item \emph{[II] Actual pair-arithmetic output}: the exact expansion of
  $\tau = xy+c$ by error-free transformations, run through exactly one full
  \TS{} sweep, keeping the leading $n$ words. This is pair arithmetic itself
  with the trailing renormalization stage removed, and gives the distribution of
  overlap that a renormalizer really receives.
\end{itemize}

For each sample we checked the following three items against exact arithmetic.
\emph{Non-overlap} counts as a violation any case where
$\max_k |z_{k+1}|/(\tfrac12\ulp(z_k))$ (the ``non-overlap degree'' below)
exceeds $1$. \emph{Sum preservation} is measured by
$\rho=|\sum_k z_k - S|/(u^K|S|)$ against the exact input sum $S=\sum_i v_i$.
$S$ is built with $4096$-bit MPFR arithmetic, and we verify that the ternary
value of every operation is $0$, i.e.\ that no rounding occurred anywhere, so
$S$ is exact rather than approximate. In addition, for every gate of the \FTS{}
chain we separately recorded whether the exponent condition $e_a \ge e_b$ and
the magnitude condition $|a|\ge|b|$ hold, and whether $s+e=a+b$ holds exactly
(error-freeness). Counting the two conditions separately is the point, since
Proposition~\ref{prop:sum} requires only the exponent condition.

Each condition was run $200{,}000$ times ($50{,}000$ for the cost-equivalent
comparison and $200$ for the accumulation loop). The environment is Ubuntu
24.04 on aarch64 (NVIDIA GB10, Cortex-X925 $\times$ 10 $+$ Cortex-A725 $\times$
10), gcc 13.3.0, CUDA 13.0 and MPFR 4.2.2 / GMP 6.3.0; CPU measurements are
pinned to the ten X925 cores.

\subsection{Configuration needed to restore non-overlap}

\begin{table}[t]
\centering
\caption{Non-overlap violations per renormalizer (input~[I], $n=K+1$, worst
case over $g\in\{53,42,31,20\}$, out of $200{,}000$ trials).}
\label{tab:quasi}
\scriptsize
\setlength{\tabcolsep}{3.5pt}
\resizebox{\columnwidth}{!}{%
\begin{tabular}{lrrrrrr}
\toprule
Renormalizer & \multicolumn{2}{c}{$K=2$} & \multicolumn{2}{c}{$K=3$}
         & \multicolumn{2}{c}{$K=4$} \\
\cmidrule(lr){2-3}\cmidrule(lr){4-5}\cmidrule(lr){6-7}
 & flops & viol. & flops & viol. & flops & viol. \\
\midrule
none                                  & 0  & 200000 & 0  & 200000 & 0  & 200000 \\
\VSK{} (Mukunoki--Ozaki)              & 6  & 0      & 12 & 200000 & 18 & 200000 \\
\RP{} without the \VS{} sweep         & 4  & 0      & 7  & 200000 & 10 & 200000 \\
$\RP{}, r=1$                          & 16 & 0      & 25 & \textbf{0} & 34 & 200000 \\
$r=1$, chain replaced by \VSK{}-style & 19 & 0      & 31 & 0      & 43 & 199990 \\
$\mathbf{\RP{},\ r=2}$                & 19 & 0      & 31 & 0      & 43 & \textbf{0} \\
$r=1$, two \VS{} sweeps               & 28 & 0      & 43 & 0      & 58 & 0 \\
\midrule
\multicolumn{7}{l}{\emph{Reference}: $n=K$ with the step-(1) guard removed
(Remark~\ref{rem:guard})} \\
\quad non-overlap violations          & 3  & 0      & 6  & 200000 & 9  & 200000 \\
\quad $\log_2 \rho$ (sum error)       &    & $87$   &    & $120$  &    & $153$ \\
\quad the same, guard in place        &    & exact  &    & exact  &    & exact \\
\bottomrule
\end{tabular}}
\end{table}

Table~\ref{tab:quasi} shows three things; all of them are facts observed over
the tested input set, not proofs. First, prefixing the \VS{} sweep is
necessary: without it every trial fails for $K\ge3$. Second, one sweep plus one
round of the \FTS{} chain suffices only up to $K\le3$; for $K=4$ non-overlap is
not restored. Third, the cheapest fix for $K=4$ is to run the \FTS{} chain
twice. Spending the same $6(K-1)$ flops on one round of a \VSK{}-style \TS{}
chain fails in $199{,}990/200{,}000$ cases, whereas two rounds of the \FTS{}
chain fail in $0/200{,}000$. Using two \VS{} sweeps also fixes it, but costs
15 flops more.

\subsection{In what sense the \FTS{} hypothesis holds}

The distinction between the exponent and magnitude conditions made in
Section~\ref{sec:guarantee} matters here. Counting the two separately on
input~[II] ($n=K+1$), for QW ($K=4$) the magnitude condition $|a|\ge|b|$ fails
$65{,}388$ times out of $600{,}000$ even with the \VS{} sweep prefixed.
Error-freeness nevertheless never breaks, because the failures are of the
Sterbenz type with $e_a=e_b$, for which the exponent condition still holds. In
other words, over input~[II] one sweep did establish the hypothesis needed for
the sum relation. This is a fact about the measured inputs, not a guarantee for
arbitrary input (Remark~\ref{rem:notorder}). The failure of $r=1$ at $K=4$ is
not caused by the hypothesis failing but because one bottom-up round of the
\FTS{} chain does not propagate the correction all the way up; two rounds
resolve it (non-overlap violations $458 \to 0$). Dropping the \VS{} sweep, on
the other hand, not only raises the magnitude-condition violations to
$81{,}908/400{,}000$ for TW and $89{,}930/600{,}000$ for QW, but breaks
error-freeness itself about $18{,}200$ times for both TW and QW, so the sum
relation of Proposition~\ref{prop:sum} no longer holds (non-overlap violations
reach $15{,}235$ for TW and $98{,}609$ for QW).

Retaining one extra word barely changes the violation count ($386$ versus $458$
out of $200{,}000$ for QW with $r=1$), and the error with respect to $\tau$ was
identical for all renormalizers. Renormalization only restores the structure;
the information already discarded by pair arithmetic does not come back.

\subsection{Cost-equivalent comparison with \VSK{}}
\label{sec:vscmp}

\begin{table}[t]
\centering
\caption{Non-overlap violations ($n=K$, $50{,}000$ trials each). $g$ is the
exponent gap of the synthetic model; smaller means deeper overlap.}
\label{tab:vscmp}
\scriptsize
\setlength{\tabcolsep}{4pt}
\resizebox{\columnwidth}{!}{%
\begin{tabular}{lrrrrr}
\toprule
Renormalizer & flops & input~[II] & $g{=}42$ & $g{=}31$ & $g{=}20$ \\
\midrule
\multicolumn{6}{l}{\emph{$K=3$ (QTW of Mukunoki--Ozaki)}} \\
\quad \textsc{VecSum3} (Mukunoki--Ozaki) & 12 & $9{,}345$ & 0 & 63 & $50{,}000$ \\
\quad $\mathbf{\RP{},\ r=1}$ (proposed) & \textbf{18} & \textbf{0} & \textbf{0}
      & \textbf{0} & \textbf{0} \\
\quad \textsc{VecSum3} $\times 2$    & 24 & 0 & 0 & 0 & 0 \\
\quad \VS{} $\to$ \textsc{VecSum3}   & 24 & 0 & 0 & 0 & 0 \\
\midrule
\multicolumn{6}{l}{\emph{$K=4$ (QQW)}} \\
\quad \textsc{VecSum4} (same form)   & 18 & $22{,}257$ & 1 & $50{,}000$ & $50{,}000$ \\
\quad $\RP{},\ r=1$                  & 27 & 69 & $49{,}943$ & $49{,}019$ & 0 \\
\quad \textsc{VecSum4} $\times 2$    & 36 & 0 & 0 & 0 & 273 \\
\quad \VS{} $\to$ \textsc{VecSum4}   & 36 & $7{,}482$ & 0 & 467 & $49{,}997$ \\
\quad $\mathbf{\RP{},\ r=2}$ (proposed) & \textbf{36} & \textbf{0} & \textbf{0}
      & \textbf{0} & \textbf{0} \\
\bottomrule
\end{tabular}}
\end{table}

Table~\ref{tab:vscmp} lists the violation counts on the same input sets with
the cost matched. Two things follow. For $K=3$ the proposal is $25\,\%$
cheaper: \textsc{VecSum3} alone (12 flops) fails on $18.7\,\%$ of input~[II]
and cannot restore non-overlap at all under deep overlap, and reaching zero
everywhere with \textsc{VecSum3} costs 24 flops, whereas the proposal reaches
zero everywhere with 18. For $K=4$, at equal cost only the proposal reaches
zero everywhere: given 36 flops, \textsc{VecSum4} $\times 2$ still fails 273
times under the deepest overlap and \VS{} $\to$ \textsc{VecSum4} fails
$49{,}997$ times, while $\RP{}, r=2$ fails in none.

A concrete counterexample is also available. For the sample of input~[II] with
$K=3$
\begin{center}\scriptsize
\resizebox{\columnwidth}{!}{$
v = (-\texttt{0x1.b6546d89bfd6ap+3},\ -\texttt{0x1.d9364cp-52},\
-\texttt{0x1.f7a8ced70c4fep-51})$}
\end{center}
the output of \textsc{VecSum3} has non-overlap degree $1.4458 > 1$: the second
word exceeds $\tfrac12\ulp$ of the first by $44\,\%$, so the three words do not
hold $159$ bits. The proposal returns non-overlap degree $0.5542$ on the same
input. Both preserve the sum exactly; only the quality of the representation
differs.

\subsection{Accumulation loop}

With the same placement that Mukunoki and Ozaki use in CG, namely one
renormalization right after the AXPY, we chained the pair multiply--add that
generates input~[II] $M$ times, renormalizing after each, and measured the
non-overlap degree of the output (worst case over $200$ trials each). With
\VSK{} the overlap grows without bound as the loop proceeds: for $K=3$ and
$M=10,100,1000$ it reaches $26.3$, $244$ and $853$, and for $K=4$ with $M=1000$
it reaches $1.58\times10^6$, i.e.\ more than three words' worth of significand
capacity has been lost to overlap. The proposal keeps $1.00$ for every
$M$. The mechanism behind the ``no convergence without normalization'' of
\cite{MukunokiOzaki25} thus survives, in weakened form, even \emph{with}
renormalization in place.

\section{Benchmark Tests}
\label{sec:bench}

\subsection{Measurement setup}

\paragraph*{Solvers and word count}
CG follows Algorithm~1 of \cite{MukunokiOzaki25} as is. The coefficient matrix
$A$ and the right-hand side $\vb$ are in the base format, while the solution
$\vx$ and every vector and scalar inside CG are $K$-word. BiCGStab (van der
Vorst) was added as a representative product-type Krylov method. It performs
about twice the arithmetic of CG and has two residual-like vectors $\vr$ and
$\vs$, so the placement corresponding to ``renormalize $\vr$ once right after
the AXPY'' of Mukunoki and Ozaki was taken to be ``renormalize $\vs$ and $\vr$
once per iteration each.'' No preconditioner is used. The base formats are
binary64 and binary32 and the word counts are $K=2,3,4$ (QDW/QTW/QQW); in
binary32 these provide nominal significand capacities of $48/72/96$ bits, so
$K=3$ already exceeds plain FP64.

\paragraph*{Test problems}
In addition to synthetic problems (where $A\vx^\ast=\vb$ is exact by
construction), $25$ matrices were taken from SuiteSparse \cite{Davis11} and
rounded to a per-row quantum following Ozaki and Ogita \cite{OzakiOgita17}, so
that each row sum is exact in the base format. We then screened them by whether
pair arithmetic actually solves them without a preconditioner (attained
relative error $\varepsilon^\ast<10^{-20}$, or $10^{-10}$ in binary32).
For binary64/binary32, $11$/$12$ matrices passed for BiCGStab and $4$/$5$ for
CG. Adding \texttt{lap2d:128} and \texttt{lap3d:40} gives $13$/$14$ problems for
BiCGStab and $6$/$7$ for CG. The $13$ matrices that passed in neither base
format were structural-analysis problems that did not converge without a
preconditioner under our test conditions, and were therefore excluded.

\paragraph*{Measured quantities}
$\varepsilon^\ast$ is the minimum of $\|\vx_k-\vx^\ast\|_2/\|\vx^\ast\|_2$ over
the iterations, ovl is the worst non-overlap degree of $\vr$, and bad is the
fraction of elements still violating non-overlap right after renormalization.
Times are best-of runs over $100$ iterations, and all kernels were checked
against $4096$-bit MPFR.

\subsection{In successful trials the words agree with \VSK{}}

What CG actually uses is the case $n=K$, where both renormalizers preserve the
sum exactly (Proposition~\ref{prop:sum}, with no $\delta_j$). However,
\emph{a non-overlapping $K$-word representation of a given exact sum is not
unique}: for $S=1+2^{-53}$, both $(1,\,2^{-53})$ and $(1+2^{-52},\,-2^{-53})$
are non-overlapping and sum to $S$. One therefore cannot argue ``the same sum
implies the same words''; agreement has to be checked experimentally.
Generating the input shape the CG residual update actually produces (one pair
addition of two non-overlapping $K$-word numbers) two million times per
setting, both renormalizers restored non-overlap in every case for $K=3,4$ in
binary64 and the outputs agreed bitwise in all of them. In binary32, \VSK{}
failed in $1$ case for $K=3$ and $15$ for $K=4$ and \RP{} in $1$ case, but the
outputs agreed in every remaining case where both succeeded. What follows rests
on this \emph{observation}.

Identical outputs imply identical trajectories, so the iteration counts and
attained accuracies agree in CG. For \texttt{lap2d:128}, binary64 and QTW,
\VSK{}, \RP{}$r{=}1$ and \RP{}$r{=}2$ all need $465$ iterations to reach
$\varepsilon^\ast\le10^{-32}$ and all attain
$\varepsilon^\ast=5.92\times10^{-46}$.

\subsection{Differences appear only where \VSK{} breaks}

\begin{table}[t]
\centering
\caption{Summary of the whole sweep, and the cases with the largest
differences. The population is $2$ base formats $\times$ $2$ solvers $\times$
problems $\times$ $K(=2,3,4)$, i.e.\ $120$ cells, times $2$ renormalization
placements, i.e.\ $240$ pairs.}
\label{tab:sweep}
\scriptsize
\setlength{\tabcolsep}{4pt}
\begin{tabular}{lr}
\toprule
 & pairs \\
\midrule
neither converged & $103$ \\
only \VSK{} converged & $2$ \\
only \RP{} converged & $0$ \\
both converged, identical iterations and $\varepsilon^\ast$ & $\mathbf{97}$ \\
both converged, \RP{} better & $22$ \\
both converged, \RP{} worse & $16$ \\
\bottomrule
\end{tabular}
\\[6pt]
\resizebox{\columnwidth}{!}{%
\begin{tabular}{llrrrr}
\toprule
Problem (BiCGStab, binary64) & $K$ & \multicolumn{2}{c}{\VSK{}}
 & \multicolumn{2}{c}{\RP{}} \\
\cmidrule(lr){3-4}\cmidrule(lr){5-6}
 & & iter. & $\varepsilon^\ast$ & iter. & $\varepsilon^\ast$ \\
\midrule
\texttt{raefsky2}  & 3 & $2093$ & $3.85\times10^{-25}$
                   & $\mathbf{1270}$ & $\mathbf{1.02\times10^{-33}}$ \\
\texttt{lap2d:128} & 3 & $570$ & $2.86\times10^{-27}$
                   & $\mathbf{555}$ & $\mathbf{1.73\times10^{-32}}$ \\
\texttt{jnlbrng1}  & 3 & $326$ & $2.43\times10^{-40}$
                   & $\mathbf{298}$ & $2.34\times10^{-40}$ \\
\bottomrule
\end{tabular}}
\end{table}

Table~\ref{tab:sweep} summarizes the whole sweep. Convergence differed in only
$2$ pairs, and in both of them it was \VSK{} that converged: \textbf{a stronger
renormalizer does not increase the number of solvable problems.} Most of the
$38$ pairs that differ do so in the last one or two digits, but some cases
matter a great deal. The win--loss bias itself is not statistically significant
(sign test, $p=0.42$), and since the same matrix recurs with different $K$,
base format and placement, whether the $38$ pairs may be treated as independent
samples is also open to question; what follows should be read as effect sizes
and individual cases. Structurally the advantage of the proposal is consistent:
of the $240$ rows of the sweep, \VSK{} produced a non-overlap violation in $66$
and \RP{} in $37$. In binary32, \VSK{} breaks even at $K=3$
(\texttt{lap2d:128}, BiCGStab: ovl-out $=40.4$, violation rate $0.0002\,\%$,
against $2.00$ and $0\,\%$ for \RP{}); at $K=4$ the ovl-out of \VSK{} is $10.2$
against $1.00$ for \RP{}$r{=}2$.

\begin{observation}
\label{obs:input}
For the CG residual update alone, the input reaching the renormalizer is an
overlap of about $26$ bits with the ordering preserved, which stays within the
range \VSK{} handles. The \VS{} sweep prefix and the tail fold of \RP{} pay off
only where the ordering of the input breaks down: binary32, and BiCGStab, where
two residual-like vectors make the cancellation deeper.
\end{observation}

\subsection{Execution time}

\begin{table*}[t]
\centering
\caption{$100$ CG iterations (\texttt{lap2d:128}, $n=16{,}384$, binary64,
Cortex-X925 $\times$ 10, best of seven round-robin runs, in seconds).
$\vr$-AXPY$+$N is the residual-update stage including renormalization;
$^\ast$ marks the configuration of Mukunoki and Ozaki. Plain FP64 takes
$0.0023$ s, which is the reference of the ``vs.\ FP64'' column.}
\label{tab:time}
\footnotesize
\setlength{\tabcolsep}{3pt}
\begin{tabular}{llrrrrrrr}
\toprule
Arithmetic & Renormalizer & flops & total & SpMV & DOT & $\vr$-AXPY$+$N
 & vs.\ FP64 & vs.\ \VSK{} \\
\midrule
QDW  & \textsc{VecSum2}$^\ast$ & 6 & $0.0047$ & $0.0009$ & $0.0022$ & $0.0007$ & $2.05$ & $1.000$ \\
     & \RP{}$r{=}1$           & 9  & $0.0048$ & $0.0009$ & $0.0022$ & $0.0008$ & $2.09$ & $1.018$ \\
\midrule
QTW  & \textsc{VecSum3}$^\ast$ & 12 & $0.0069$ & $0.0021$ & $0.0021$ & $0.0011$ & $2.99$ & $\mathbf{1.000}$ \\
     & \textsc{VecSum3}$\times2$ & 24 & $0.0071$ & $0.0021$ & $0.0021$ & $0.0013$ & $3.09$ & $1.032$ \\
     & \RP{}$r{=}1$           & 18 & $0.0071$ & $0.0021$ & $0.0021$ & $0.0014$ & $3.09$ & $\mathbf{1.034}$ \\
     & \RP{}$r{=}2$           & 24 & $0.0073$ & $0.0021$ & $0.0021$ & $0.0015$ & $3.16$ & $1.057$ \\
\midrule
QQW  & \textsc{VecSum4}$^\ast$ & 18 & $0.0123$ & $0.0039$ & $0.0030$ & $0.0022$ & $5.41$ & $1.000$ \\
     & \RP{}$r{=}1$           & 27 & $0.0128$ & $0.0040$ & $0.0030$ & $0.0027$ & $5.62$ & $1.038$ \\
     & \RP{}$r{=}2$           & 36 & $0.0132$ & $0.0039$ & $0.0030$ & $0.0031$ & $5.78$ & $1.068$ \\
\bottomrule
\end{tabular}
\end{table*}

Table~\ref{tab:time} gives the time for $100$ CG iterations. Measured on its
own (length $16{,}384$, binary64), \textsc{VecSum3} takes $3.32\,\mu$s against
$4.38\,\mu$s for \RP{}$r{=}1$, a factor of $1.32$, smaller than the flop ratio
$18/12=1.50$; for the whole CG solver the overhead is only $+3.4\,\%$, and it
vanishes as the SpMV becomes the bottleneck ($+0.7\,\%$ on \texttt{pdb1HYS},
whose nnz$/n$ is $119.3$).

\subsection{SIMD and accelerators}
\label{sec:simd}

Since neither pair arithmetic nor the proposed renormalizer contains a branch,
SIMD vectorization requires only replacing the word type by a vector type,
without touching the source of the arithmetic itself (AoSoA layout; the SpMV
uses one ELLPACK row per lane). With $128$-bit NEON/SVE2 this gives two lanes
in binary64 and four in binary32, and the speed-ups of SpMV/DOT/AXPY at
$n=2^{20}$ on a single thread were $1.69$--$3.73\times$ and
$2.01$--$7.25\times$ for pair arithmetic. Exceeding the lane count is presumably
the scalar version being limited by its dependence chain, but we did not
separate instruction throughput from code-generation effects, so we do not
claim it. The ranking is unchanged. The same headers compile for CUDA as they
are, with \texttt{--fmad=false}. On GB10 ($n=2^{22}$, ELLPACK row length $5$)
the narrow FP64 path makes binary64 time track the operation count directly
($4.82$ ms for pair arithmetic versus $8.82$ ms for the $92$-flop BF at $K=3$),
whereas in binary32 all three families become memory bound and land within
$1.11$--$1.12$ ms of each other.

\begin{observation}
\label{obs:gpu}
Since the motivation for pair arithmetic is to reduce the operation count, its
benefit disappears once the kernel is memory bound, as in this experiment.
Conversely, on GB10 a binary32 TW ($72$ bits) is both more accurate than plain
FP64 and faster than FP64-based multiple precision, which is one answer to the
question raised as future work in \cite{MukunokiOzaki25}.
\end{observation}

\subsection{Placement and frequency of renormalization}

\begin{table}[t]
\centering
\caption{Placement of renormalization (\texttt{lap2d:128}, binary64, QTW,
$3{,}000$ iterations). Times are for $100$ iterations with $\vr$ only taken as
$1.00$; $\varepsilon^\ast$ agrees between \VSK{} and \RP{}.}
\label{tab:place}
\scriptsize
\setlength{\tabcolsep}{4pt}
\resizebox{\columnwidth}{!}{%
\begin{tabular}{lrrr}
\toprule
Placement & $\varepsilon^\ast$ & time (\VSK{}) & time (\RP{}) \\
\midrule
none        & $3.46\times10^{-21}$ & $0.97$ & $0.94$ \\
$\vr^\ast$  & $5.92\times10^{-46}$ & $1.00$ & $1.00$ \\
$\vr{+}\vp$     & $4.99\times10^{-46}$ & $1.04$ & $1.07$ \\
$\vr{+}\vq$     & $5.94\times10^{-46}$ & $1.04$ & $1.06$ \\
$\vr{+}\vx$     & $9.05\times10^{-47}$ & $1.03$ & $1.06$ \\
$\vr{+}\vp{+}\vx$ & $1.16\times10^{-47}$ & $1.07$ & $1.12$ \\
$\mathbf{\vr{+}\vp{+}\vq{+}\vx}$ & $\mathbf{7.86\times10^{-48}}$ & $\mathbf{1.10}$ & $1.19$ \\
inside DOT as well & $5.04\times10^{-48}$ & $1.57$ & $1.63$ \\
inside SpMV as well & $4.25\times10^{-48}$ & $2.23$ & $2.30$ \\
\bottomrule
\end{tabular}}
\end{table}

Table~\ref{tab:place} sweeps the placement. The solution vector $\vx$ is
never renormalized in the implementation of Mukunoki and Ozaki, yet
$\vx \leftarrow \vx+\alpha\vp$ is a pure accumulation repeated hundreds of
times, so adding a single renormalization there improves $\varepsilon^\ast$ by
$6.5\times$, and together with $\vp$ and $\vq$ by $75\times$, at a cost of
$+10\,\%$. By contrast, renormalizing inside the $O(\mathit{nnz})$ SpMV and the
$O(n)$ DOT costs $2.2\times$ the time for only $1.8\times$ the accuracy; this is
where insertions should be removed.

Lowering the \emph{frequency} of renormalizing $\vr$ to once every $m$
iterations is not an option. Setting $m=2$ alone costs $43$ orders of magnitude
in $\varepsilon^\ast$ ($5.92\times10^{-46} \to 5.74\times10^{-3}$): during one
un-renormalized iteration the overlap of $\vr$ grows from $26$ to $71$ bits,
$\rho=\vr^T\vr$ is destroyed and conjugacy is lost. \RP{} still protects the
structure here (bad-out below $0.07\,\%$ at $m=2$--$8$, against
$65$--$95\,\%$ for \VSK{}), but the bits discarded by pair arithmetic in the
un-renormalized iterations do not come back, so this does not translate into
accuracy.

\subsection{BiCGStab: once per iteration is not enough}

\begin{table}[t]
\centering
\caption{Number of converged problems and geometric means over the $13$
BiCGStab problems in binary64 ($\varepsilon^\ast$; times relative to \VSK{}
$=1.00$).}
\label{tab:fambicg}
\scriptsize
\setlength{\tabcolsep}{4pt}
\resizebox{\columnwidth}{!}{%
\begin{tabular}{llrrr}
\toprule
$K$ & Arithmetic and renormalization & conv. & $\varepsilon^\ast$ & time \\
\midrule
3 & pair, no renormalization        & $2/13$  & $6.00\times10^{-12}$ & $1.02$ \\
3 & pair $+$ \textsc{VecSum3} (M-O) & $7/13$  & $1.27\times10^{-24}$ & $1.00$ \\
3 & pair $+$ \RP{}                  & $7/13$  & $1.01\times10^{-25}$ & $1.03$ \\
3 & pair $+$ \VSK{}, $\vr{+}\vp\vq\vx$ & $8/13$  & $3.33\times10^{-35}$ & $1.04$ \\
3 & pair $+$ \RP{} every operation  & $\mathbf{13/13}$ & $8.12\times10^{-46}$ & $4.56$ \\
3 & BF FMA (renormalized every op.) & $\mathbf{13/13}$ & $7.20\times10^{-46}$ & $\mathbf{2.68}$ \\
\midrule
4 & pair $+$ \textsc{VecSum4} (M-O) & $6/13$  & $9.70\times10^{-31}$ & $1.00$ \\
4 & pair $+$ \VSK{}, $\vr{+}\vp\vq\vx$ & $8/13$  & $2.53\times10^{-44}$ & $1.04$ \\
4 & pair $+$ \RP{} every operation  & $\mathbf{13/13}$ & $1.64\times10^{-59}$ & $4.72$ \\
4 & BF FMA (renormalized every op.) & $\mathbf{13/13}$ & $4.74\times10^{-60}$ & $\mathbf{2.50}$ \\
\midrule
2 & pair $+$ \textsc{VecSum2} (M-O) & $6/13$  & $2.08\times10^{-17}$ & $1.00$ \\
2 & pair $+$ \RP{} every operation  & $11/13$ & $1.56\times10^{-30}$ & $4.77$ \\
2 & BF FMA (renormalized every op.) & $\mathbf{11/13}$ & $\mathbf{9.55\times10^{-31}}$ & $\mathbf{1.32}$ \\
\bottomrule
\end{tabular}}
\end{table}

Table~\ref{tab:fambicg} gives the number of converged problems for BiCGStab.
binary32 behaves the same way: out of $14$ problems, \VSK{} in the M-O
placement converges on $2$ irrespective of $K$, the four-point placement
reaches only $5$--$6$, and renormalizing every operation reaches $8$--$10$.

\begin{observation}
\label{obs:wall}
For BiCGStab the prescription of Mukunoki and Ozaki, one renormalization per
iteration, does not carry over. Neither a stronger renormalizer (\RP{}) nor
more placements ($\vr,\vp,\vq,\vx$) moves the count beyond $6$--$8/13$. The only
configuration that reached $13/13$ in this experiment was renormalizing every
operation, i.e.\ abandoning pair arithmetic, and the cheapest way to do that is
BF FMA \cite{KouyaFMA26}, a branch-free FMA ($2.50$--$2.68\times$ the pair
cost), not pair $+$ \RP{} at every operation ($4.56$--$4.72\times$). For $K=2$,
BF FMA raises the count from $6/13$ to $11/13$ at $1.32\times$.
\end{observation}

Observation~\ref{obs:wall} delimits the scope of the proposal. The information
discarded by an operation that was not renormalized does not come back with any
renormalizer. What the proposal answers is how strong one renormalization can
be made; that does pay off where \VSK{} breaks, but it does not overcome the
constraint on frequency itself.

\section{Conclusion and Future Work}

We proposed \RP{}, a branch-free fixed-trip-count renormalizer for returning a
pair-arithmetic output to a non-overlapping $K$-word form, given as a single
construction parameterized by the word count $K$ and the input length $n$. It
consists of three stages---a \VS{} sweep, a tail fold and $r$ rounds of a
\FTS{} chain---and costs $6(n-1)+(n-K)+3r(K-1)$ flops. None of the three can be
dropped: without the \VS{} sweep, \FTS{} ceases to be an error-free
transformation and Proposition~\ref{prop:sum} collapses, and without the $n>K$
guard on the tail fold the lowest word is doubled. The latter is particularly
insidious because it slips through the non-overlap test. The smallest round
count that produced no violation over the range tested was $r=1$ for $K\le3$ and
$r=2$ for $K=4$; spending the same cost on a \VSK{}-style \TS{} chain fails for
$K=4$. For $K=3$ the proposal delivers the same quality $25\,\%$ more cheaply
than the \VSK{}-style chain.

The scope of what is guaranteed should be stated plainly. The hypothesis of
Proposition~\ref{prop:sum} depends on the exponent layout of the input and is
not established in general by one sweep (counterexample in
Remark~\ref{rem:notorder}). What can be stated mathematically is only a
conditional claim; for the pair-arithmetic inputs we target, the fact is that
the hypothesis held at all $600{,}000$ \FTS{} calls in the measurements. A
correctly rounded lowest word is likewise not guaranteed, because what
non-overlap supplies is a bound exceeding the rounding boundary by the factor
$1/(1-u)$, not the one bit telling which side of it the tail is on. This is a
limitation of the sweep-based construction (Proposition~\ref{prop:notcr}), not
an impossibility result for the branch-free fixed-trip-count model as such.

Integrating the scheme into iterative solvers did not produce a simple ranking.
Over two million instances of the input shape produced by the CG residual
update, the two renormalizers agreed bitwise wherever both restored
non-overlap; since a non-overlapping representation is not unique, this is an
observation rather than a theorem, and CG in binary64 cannot tell them apart.
Differences appear where \VSK{} breaks, namely binary32 and BiCGStab: of the
$135$ pairs out of $240$ in which both converged, $97$ agree exactly and the
remaining $38$ split $22$--$16$ in favour of \RP{}. That bias is not
significant (sign test, $p=0.42$), although the wins are large when they occur
(\texttt{raefsky2}: $39\,\%$ fewer iterations, eight orders of magnitude better
accuracy). The clear bias is structural: non-overlap violations fall from $66$
rows to $37$ and the improve/degrade split is $61$ to $6$ (sign test
$p<10^{-6}$ if the pairs are treated as independent). The price is $1.32\times$
on the renormalizer alone and $+3.4\,\%$ on the whole CG solver. In CG it is
cheaper to \emph{add placements} than to strengthen the renormalizer
(renormalizing $\vx,\vp,\vq$ as well improves the attained accuracy by a factor
of $75$ for $+10\,\%$), while in BiCGStab
the once-per-iteration placement itself is insufficient: the only configuration
reaching $13/13$ here renormalized every operation, most cheaply through a
branch-free FMA.

Four items remain. First, \emph{machine-checked proof:} every claim about
non-overlap here is experimental, and because pair-arithmetic inputs overlap,
the domination relations of FPANVerifier \cite{ZhangAiken25} cannot express an
assumption on magnitudes alone; extending that abstraction is the next step.
Second, \emph{the choice of $r$}: \eqref{eq:rchoice} was fixed experimentally
and extrapolation to $K\ge5$ must likewise rely on measurement. Third, \emph{the
frequency and placement of renormalization}: this paper gives only the one-sided
answer ``strengthen one renormalization without lowering the frequency,'' and by
Observation~\ref{obs:wall} the frequency constraint dominates for BiCGStab;
Table~\ref{tab:place} hints that a vector such as $\vx$, which does not feed
back into the recurrence, can be renormalized less often. Fourth,
\emph{accelerators}: by Observation~\ref{obs:gpu} the benefit of pair arithmetic
disappears in the memory-bound regime, so a selection guideline in terms of base
format, word count and accelerator is needed.

\section*{Acknowledgment}

This work was supported by JSPS KAKENHI Grant Number JP26K14846. The author
thanks Daichi Mukunoki and Katsuhisa Ozaki for discussions on quasi multi-word
arithmetic and for making their implementation public.

\section*{Use of Generative AI}

Generative AI (Anthropic Claude, Claude Code) was used to prepare the
manuscript and to run and aggregate the benchmark tests. The proposal of the
algorithm, the design of the experiments and the interpretation and
verification of the results were carried out by the author, who is responsible
for the content.


\end{document}